\documentclass[12pt]{article}
\usepackage[UKenglish]{babel}
\usepackage[T1]{fontenc}
\usepackage{lmodern,amsmath,amsthm,amsfonts,amssymb,graphicx,float,microtype,thmtools,underscore,mathtools,thm-restate}
\usepackage[shortlabels]{enumitem}
\setlist[itemize]{topsep=0ex,itemsep=0ex,parsep=0ex}
\setlist[enumerate]{topsep=0ex,itemsep=0ex,parsep=0ex}
\usepackage[usenames,dvipsnames,svgnames,table]{xcolor}
\usepackage{todonotes}
\usepackage[unicode=true]{hyperref}
\hypersetup{ 
colorlinks,
breaklinks,
linkcolor={blue!60!black},
citecolor={black},
urlcolor={blue!60!black},
pdftitle={Clustered Colouring of Odd-$H$-Minor-Free Graphs}}
\usepackage[capitalise, compress, nameinlink, noabbrev]{cleveref}
\crefname{lem}{Lemma}{Lemmas}
\crefname{thm}{Theorem}{Theorems}
\crefname{ques}{Question}{Theorems}
\crefname{cor}{Corollary}{Corollaries}
\crefname{enumi}{Item}{Items}
\crefformat{equation}{(#2#1#3)}
\Crefformat{equation}{Equation #2(#1)#3}
\crefformat{enumi}{#2#1#3}
\Crefformat{enumi}{Item (#2#1#3)}
\newcommand{\defn}[1]{\textcolor{Maroon}{\emph{#1}}}
\usepackage[longnamesfirst,numbers,sort&compress]{natbib}
\makeatletter
\def\NAT@spacechar{~}
\makeatother
\usepackage[margin=27mm]{geometry}
\renewcommand{\baselinestretch}{1.075}
\allowdisplaybreaks
\DeclarePairedDelimiter{\abs}{\lvert}{\rvert}

\renewcommand{\epsilon}{\varepsilon}
\renewcommand{\emptyset}{\varnothing}
\renewcommand{\geq}{\geqslant}
\renewcommand{\leq}{\leqslant}
\DeclareMathOperator{\dist}{dist}
\DeclareMathOperator{\tw}{tw}
\DeclareMathOperator{\td}{td}

\newcommand{\FF}{\mathcal{F}}
\newcommand{\GG}{\mathcal{G}}

\newcommand{\NN}{\mathbb{N}}

\newcommand{\WW}{\mathcal{W}}
\newcommand{\odd}{\text{odd}}

\DeclareMathOperator{\ctd}{\overline{td}}
\theoremstyle{plain}
\newtheorem{thm}{Theorem}
\newtheorem{lem}[thm]{Lemma}

\crefname{obs}{Observation}{Observations}
\newtheorem*{lem*}{Lemma}
\theoremstyle{definition}
\newtheorem{conj}[thm]{Conjecture}
\newtheorem*{conj*}{Conjecture}
\presetkeys%
{todonotes}%
{inline}{}

\begin{document}
\title{\bf\boldmath\fontsize{18pt}{18pt}\selectfont Clustered Colouring of Odd-$H$-Minor-Free Graphs}

\author{%
Robert~Hickingbotham\,\footnotemark[4] \qquad 
Dong Yeap Kang\,\footnotemark[5]  \qquad 
Sang-il Oum\,\footnotemark[3]  \\
Raphael~Steiner\,\footnotemark[2] \qquad
David~R.~Wood\,\footnotemark[4]}
 
\footnotetext[4]{School of Mathematics, Monash University, Melbourne, Australia (\texttt{\{robert.hickingbotham, david.wood\}@monash.edu}). Research of R.H.\ supported by an Australian Government Research Training Program Scholarship. Research of D.W.\ supported by the Australian Research Council.}

\footnotetext[5]{Extremal Combinatorics and Probability Group, Institute for Basic Science, Daejeon, South Korea (\texttt{dykang.math@ibs.re.kr}). Research of DY.K. supported by the Institute for Basic Science (IBS-R029-Y6).}

\footnotetext[3]{Discrete Mathematics Group, Institute for Basic Science, Daejeon, South Korea; and Department of Mathematical Sciences, KAIST, Daejeon, South Korea (\texttt{sangil@ibs.re.kr}). Research of S.O. supported by the Institute for Basic Science (IBS-R029-C1).}

\footnotetext[2]{Institut für Theoretische Informatik, ETH Zürich, Switzerland  (\texttt{raphaelmario.steiner@inf.ethz.ch}). Research of R.S.\ supported by an ETH Zurich Postdoctoral Fellowship.}

\maketitle

\begin{abstract}
The clustered chromatic number of a graph class $\GG$ is the minimum integer $c$ such that every graph $G\in\GG$ has a $c$-colouring where each monochromatic component in $G$ has bounded size. We study the clustered chromatic number of graph classes $\GG_H^{\odd}$ defined by excluding a graph $H$ as an odd-minor. How does the structure of $H$ relate to the clustered chromatic number of $\GG_H^{\odd}$? We adapt a proof method of Norin, Scott, Seymour and Wood (2019) to show that the clustered chromatic number of $\GG_H^{\odd}$ is tied to the tree-depth of $H$.
\end{abstract}

\section{Introduction}
We consider simple, finite, undirected graphs $G$ with vertex-set ${V(G)}$ and edge-set ${E(G)}$. See~\citep{diestel2017graphtheory} for graph-theoretic definitions not given here. Let ${\NN \coloneqq \{1,2,\dots\}}$ and ${\NN_0 \coloneqq \{0,1,\dots\}}$. For $n \in \NN$, let $[n]\coloneq \{1,\dots,n\}$.

A \defn{colouring} of a graph $G$ is a function that assigns one colour to each vertex of $G$. For an integer $k\geq 1$, a \defn{$k$-colouring} is a colouring using at most $k$ colours. A colouring of a graph is \defn{proper} if each pair of adjacent vertices receives distinct colours. The \defn{chromatic number $\chi(G)$} of a graph $G$ is the minimum integer $k$ such that $G$ has a proper $k$-colouring. A graph $H$ is a \defn{minor} of a graph $G$ if $H$ is isomorphic to a graph that can be obtained from a subgraph of G by contracting edges. A graph $G$ is \defn{$H$-minor-free} if $H$ is not a minor of $G$. Let \defn{$\GG_H$} denote the class of $H$-minor-free graphs. Hadwiger \cite{hadwiger1943klassifikation} famously conjectured that $\chi(G)\leq h-1$ for every $K_h$-minor-free graph $G$. This is one of the most important open problems in graph theory. The best known upper bound on the chromatic number of $K_h$-minor-free graphs is $O(h \log \log h)$ due to \citet{DP21} (see \cite{kostochka1984average,thomason1984contraction,NPS23} for previous bounds). It is open whether every $K_h$-minor-free graph is properly $O(h)$-colourable. See \cite{SeymourHC} for a survey on this conjecture.

We consider the intersection of both a relaxation and a strengthening of Hadwiger's conjecture. First, the relaxation. For $k\in \NN_0$, a colouring has \defn{clustering $k$} if each monochromatic component has at most $k$ vertices. The \defn{clustered chromatic number $\chi_{\star}(\GG)$} of a graph class~$\GG$ is the minimum $c\in \NN_0$ for which there exists $k\in \NN_0$ such that every graph $G\in \GG$ can be $c$-coloured with clustering $k$. Similarly, the \defn{defective chromatic number $\chi_{\Delta}(\GG)$} of a graph class $\GG$ is the minimum $c\in \NN_0$ for which there exists $d\in \NN_0$ such that every graph $G\in \GG$ can be $c$-coloured with each monochromatic component having maximum degree at most $d$. See \cite{wood2018defective} for a survey on clustered and defective colouring.

Motivated by Hadwiger's conjecture, \citet{EKKOS15} proved a defective variant: $\chi_{\Delta}(\GG_{K_h})=h-1$ for all $h\in \NN$. They also introduced the \defn{Clustered Hadwiger Conjecture}: for all $h\in \NN$, 
$$\chi_{\star}(\GG_{K_h})=h-1.$$ 
\citet{KM07} had previously proved the first $O(h)$ upper bound on~$\chi_{\star}(\GG_{K_h})$. The constant in the $O(h)$ term was successively improved \citep{KM07,vdHW18,EKKOS15,Norin15,LO18,Wood10,LW2,LW3}, before a proof of the Clustered Hadwiger Conjecture was first announced to appear in a future paper by \citet{DN17}. Their full proof has not yet been made public. \citet{DEMW23} recently proved the conjecture.

Now for the strengthening of Hadwiger's conjecture. Let $G$ and $H$ be graphs. An \defn{$H$-model} in $G$ is a collection $\mathcal{H}=(T_x:x\in V(H))$ of pairwise vertex-disjoint subtrees of $G$ such that for each $xy\in E(H)$, there is an edge of $G$ joining $T_x$ and $T_y$. We denote by \defn{$G\langle\mathcal{H}\rangle$} the subgraph of $G$ that is the union $\bigcup(T_x \colon x\in V(H))$ together with the edges in $G$ joining $T_x$ and $T_y$ for each $xy\in E(H)$. 
Clearly $H$ is a minor of $G$ if and only if there exists an $H$-model in $G$. We call $\mathcal{H}$ \defn{odd} if there is a $2$-colouring $c$ of $G\langle\mathcal{H}\rangle$ such that:
\begin{itemize}
    \item for each $x\in V(H)$, $T_x$ is properly $2$-coloured by $c$; and
    \item for each $xy\in E(H)$, there is a monochromatic edge in $G\langle\mathcal{H}\rangle$ joining $T_x$ and $T_y$.
\end{itemize}
We say $c$ \defn{witnesses} that $\mathcal{H}$ is odd. If there is an odd $H$-model in $G$, then we say that $H$ is an \defn{odd-minor} of $G$. A graph $G$ is \defn{$H$-odd-minor-free} if $H$ is not an odd-minor of $G$. Let \defn{$\GG_H^{\odd}$} denote the class of $H$-odd-minor-free graphs.

As a qualitative strengthening of Hadwiger's conjecture, Gerards and Seymour~\cite[pp.~115]{JT95} introduced the \defn{Odd Hadwiger Conjecture}: $\chi(G)\leq h-1$ for every $K_h$-odd-minor-free graph~$G$. Note that in the case when $h=3$, excluding $K_3$ as an odd-minor is equivalent to excluding an odd cycle as a subgraph, so the class of $K_3$-odd-minor-free graphs is exactly the class of all bipartite graphs.
Recently, \citet{Steiner22} showed that the Odd Hadwiger Conjecture is asymptotically equivalent to Hadwiger's conjecture, in the sense that if every $K_t$-minor-free graph is properly $f(t)$-colourable, then every $K_t$-odd-minor-free graph is properly $2\,f(t)$-colourable. 

We are interested in understanding the clustered chromatic number of graph classes defined by a forbidden odd-minor $H$. How does the structure of $H$ relate to the clustered chromatic number of $\GG_H^{\odd}$? \citet{Kawa08} first proved
$\chi_{\star}(\GG_{K_h}^{\odd}) \in O(h)$. The constant factor term in the $O(h)$ was improved in \cite{KO19,LW2}, culminating in a result of \citet{Steiner23} who showed that $\chi_{\star}(\GG_{K_h}^{\odd}) \leq 2h-2$ using chordal partitions. So $\chi_{\star}(\GG_H^{\odd})\leq 2|V(H)|-2$ for every graph $H$. We show that the clustered chromatic number of $\GG_H^{\odd}$ is tied to the connected tree-depth of $H$, which we now define.

The \defn{vertex-height} of a rooted tree $T$ is the maximum number of vertices on a path from a root to a leaf in $T$. For two vertices $u$, $v$ in a rooted tree $T$, we say that $u$ is a \defn{descendant} of~$v$ in $T$ if $v$ lies on the path from the root to $u$ in $T$. The \defn{closure} of $T$ is the graph with vertex set $V(T)$ where $uv\in E(T)$ if $u$ is a descendant of $v$ or vice versa. The \defn{connected tree-depth~$\ctd(G)$} of a graph $G$ is the minimum vertex-height of a rooted tree $T$ with $V(T)=V(G)$ such that $G$ is a subgraph of the closure of $T$\footnote{A forest is \defn{rooted} if each component has a root vertex (which defines the ancestor relation). The \defn{closure} of a rooted forest $F$ is the disjoint union of the closure of each rooted component of $F$. The \defn{tree-depth $\td(G)$} of a graph $G$ is the minimum vertex-height of a rooted forest $F$ with $V(F)=V(G)$ such that $G$ is a subgraph of the closure of $F$. If $G$ is connected, then $\td(G)=\ctd(G)$. In fact, $\td(G)=\ctd(G)$ unless $G$ has two connected components $G_1$ and $G_2$ with $\td(G_1)=\td(G_2)=\td(G)$, in which case $\ctd(G)=\td(G)+1$.}.

\citet{OOW2019defective} showed that $\ctd(H)$ is a lower bound for the defective chromatic number of $\GG_H$. In particular, for every graph $H$, 
$$\ctd(H)-1 \leq \chi_{\Delta}(\GG_H).$$
\citet{OOW2019defective} conjectured that this inequality holds with equality. 
A partial answer to this conjecture was given by 
\citet{NSSW19}, who showed that the defective chromatic number of $\GG_H$, the clustered chromatic number of $\GG_H$, and the connected tree-depth of $H$ are tied. In particular, 
\begin{equation}
    \label{NSSW}
    \ctd(H)-1 \leq 
 \chi_{\Delta}(\GG_H) \leq
 \chi_{\star}(\GG_H)\leq 2^{\ctd(H)+1}-4.
 \end{equation}
\citet{NSSW19} also gave examples of graphs $H$ such that $\chi_{\star}(\GG_H)\geq 2\ctd(H)-2$, and conjectured that $\chi_{\star}(\GG_H)\leq 2\ctd(H)-2$ for every graph $H$. Recently, \citet{Liu24} proved the above conjecture of \citet{OOW2019defective} which, together with a result in \citet{LO18}, implies that for every graph $H$, 
 $$\ctd(H)-1 = \chi_{\Delta}(\GG_H)
\leq \chi_{\star}(\GG_H)\leq 3 \chi_{\Delta}(\GG_H)\leq 3\ctd(H)-3.$$

Now consider the defective and clustered chromatic numbers of $\GG_H^{\odd}$. We prove the following analogue of the above result of \citet{NSSW19} for odd-minors. 

\begin{thm}
\label{MainExponential}
For every graph $H$, 
$$\chi_{\star}(\GG_H^{\odd}) \leq 3\cdot 2^{\ctd(H)}- 4.$$
\end{thm}

This implies that 
$\chi_{\Delta}(\GG_H^{\odd})$ and $\chi_{\star}(\GG_H^{\odd})$ are tied to the connected tree-depth of $H$. In particular,
\begin{align*}
& \ctd(H)-1 = 
\chi_{\Delta}(\GG_H)
\leq 
\chi_{\Delta}(\GG_H^{\odd})
\leq 
\chi_{\star}(\GG_H^{\odd})
\leq 
3\cdot 2^{\ctd(H)}- 4\quad\text{and}\\
& \ctd(H)-1 = 
\chi_{\Delta}(\GG_H)
\leq 
\chi_{\star}(\GG_H)
\leq 
\chi_{\star}(\GG_H^{\odd})
\leq 
3\cdot 2^{\ctd(H)}- 4.
\end{align*}
We conjecture the following improved upper bounds.
\begin{conj}\label{MainConjecture}
    For every graph $H$, 
    $$\chi_{\star}(\GG_H^{\odd})\leq 2\ctd(H)-2 \quad\text{and}\quad \chi_{\Delta}(\GG_H^{\odd})=\ctd(H)-1.$$
\end{conj}

\citet{LW2} proved that for every graph $H$ and for all integers $s,t\in \NN$, every $H$-odd-minor-free $K_{s,t}$-subgraph-free graph is $(2s+1)$-colourable with clustering at most some function $f(H,s,t)$. With $s=1$ this says that for every graph $H$ and integer $t$, every $H$-odd-minor-free graph with maximum degree less than $t$ is 3-colourable with clustering at most some function $f(H,t)$. This implies that $\chi_{\star}(\GG_H^{\odd}) \leq 3\, \chi_{\Delta}(\GG_H^{\odd})$. 

%%%%%%%%%%%%%%%%%%%%%%%%%%%%%%%%%%%%%%%%%%%%%%%%%
\section{Proof}

We now work towards proving \cref{MainExponential}. Our proof uses the same strategy that \citet{NSSW19} used to prove \eqref{NSSW} with some modifications in the odd-minor-free setting. 

We need the following definition. A \defn{tree-decomposition} of a graph $G$ is a collection ${\WW = (W_x \colon x \in V(T))}$ of subsets of ${V(G)}$ indexed by the nodes of a tree $T$ such that
(a) for every edge ${vw \in E(G)}$, there exists a node ${x \in V(T)}$ with ${v,w \in W_x}$; and 
(b) for every vertex ${v \in V(G)}$, the set ${\{ x \in V(T) \colon v \in W_x \}}$ induces a non-empty (connected) subtree of~$T$. The \defn{width} of $\WW$ is ${\max\{ \abs{W_x} \colon x \in V(T) \}-1}$. The \defn{treewidth $\tw(G)$} of a graph $G$ is the minimum width of a tree-decomposition of $G$. Treewidth is the standard measure of how similar a graph is to a tree. Indeed, a connected graph has treewidth at most 1 if and only if it is a tree.

The following result due to \citet{DHK-SODA10} allows us to work in the bounded treewidth setting.

\begin{lem}[\cite{DHK-SODA10}]\label{TreewidthReductionLemma}
    For every graph $H$, there exists a constant $w\in \NN$ such that every $H$-odd-minor-free graph $G$ has a $2$-colouring where each monochromatic component has treewidth at most $w$.
\end{lem}

The following Erd\H{o}s-P\'{o}sa type result is folklore (see \citep[Lemma~9]{ISW} which implies this). 

\begin{lem}
\label{HittingSet}
    For every graph $G$, for every collection $\FF$ of connected subgraphs of $G$, and for every $\ell\in\NN$, either\textnormal{:}
    \begin{enumerate}[\textnormal{(}a\textnormal{)}]
        \item there are $\ell$ vertex-disjoint subgraphs in $\FF$, or
        \item there is a set $S\subseteq V(G)$ of size at most $(\ell-1)(\tw(G)+1)$ such that $S \cap V(F) \neq \emptyset$ for all $F \in \FF$.
    \end{enumerate}
\end{lem}

For $h,d\in \NN$, let \defn{$U_{h,d}$} be the closure of the complete $d$-ary tree of vertex-height $h$. Observe that $U_{h,d}$ has connected tree-depth $h$. Moreover, for every graph $H$ of connected tree-depth at most $h$, $H$ is isomorphic to a subgraph of $U_{h,d}$ for $d=|V(H)|$. So to prove \cref{MainExponential}, it suffices to do so for $U_{h,d}$ for all $h,d\in \NN$. 

The next definition is critical in allowing us to lift the proof of \citet{NSSW19} from the minor-free setting to the odd-minor-free setting. Say an $H$-model $(T_x:x\in V(H))$ in $G$ is \defn{non-trivial} if $|V(T_x)|\geq 2$ for each $x\in V(H)$. Non-trivial models have been previously studied, for example, in \citep{FJTW12}. The next lemma is the heart of the paper. 

\begin{lem}\label{MainTreewidth}
    For all $d,h,w \in \NN$, every graph $G$ with treewidth at most $w$ that does not contain a non-trivial odd $U_{h,d}$-model is $(3\cdot 2^{h-1}-2)$-colourable with clustering at most $dw+d-w$.
\end{lem}

\begin{proof}
    We proceed by induction on $h$. In the base case $h=1$, we have $U_{1,d}\cong K_1$, implying $G$ contains no tree on at least two vertices, and thus $E(G)=\emptyset$. Hence $G$ can be coloured with $1=3\cdot 2^{1-1}-2$ colour with clustering $1$. 

    Now assume $h\geq 2$ and the claim holds for $h-1$. Without loss of generality, we may assume that $G$ is connected. Fix $r\in V(G)$. For each $i\in \NN_0$, let $V_i\coloneq \{v\in V(G):\dist_G(v,r)=i\}$. We refer to each $V_i$ as a \defn{layer}. Fix $i\geq 1$ where $V_i\neq \emptyset$, and let $u_i\in V_i$. Let 
    $\mathcal{M}$ be the collection of all 
    non-trivial odd $U_{h-1,d}$-models in $G[V_i\setminus\{u_i\}]$, and let $\mathcal{F}\coloneq\{ G\langle\mathcal{H}\rangle:\mathcal{H}\in \mathcal{M}\}$. Each element of $\mathcal{F}$ is a connected subgraph of $G[V_i\setminus\{u_i\}]$ since $U_{h-1,d}$ is connected. 
    
    Suppose that $\mathcal{F}$ contains $d$ pairwise vertex-disjoint subgraphs $G\langle\mathcal{H}_1\rangle,\dots,G\langle\mathcal{H}_d\rangle$. Then each $\mathcal{H}_j$ is a non-trivial odd $U_{h-1,d}$-model in $G[V_i\setminus\{u_i\}]$. Say $\mathcal{H}_j=(T_x^{(j)}\colon x\in V(U_{h-1,d}))$. Let $c_j$ be a red--blue colouring of $G\langle\mathcal{H}_j\rangle$ that witnesses $\mathcal{H}_j$ being odd. Since $\mathcal{H}_j$ is non-trivial, $T_x^{(j)}$ contains both a red vertex and a blue vertex for each $x\in V(U_{h-1,d})$. Let $T$ be a spanning tree of $G[V_0\cup \dots \cup V_{i-1}\cup \{u_i\}]$ rooted at $r$, where for each non-root vertex $u\in V(T) \cap V_j$ there is exactly one edge $uw$ in $T$ with $w\in V(T)\cap V_{j-1}$. Then $|V(T)|\geq 2$. Moreover, there is a proper $2$-colouring $c_T$ of $T$ where each vertex in $V_{i-1}$ is coloured red. Since every vertex in $V_i$ has a neighbour in $V_{i-1}$, for each $x\in V(H)$ and $j\in [d]$, each red vertex in $T_x^{(j)}$ is adjacent to a red vertex in $T$. Observe that $U_{h,d}$ can be constructed by taking $d$ disjoint copies of $U_{h-1,d}$ plus a dominant vertex. By taking $\mathcal{H}_1,\dots,\mathcal{H}_d$ together with~$T$ and all the edges between $G\langle\mathcal{H}_1\rangle\cup\dots\cup G\langle\mathcal{H}_d\rangle$ and $V_{i-1}$, it follows that $G$ contains a non-trivial $U_{h,d}$-model $\mathcal{H}=(T_x\colon x\in V(U_{h,d}))$. Furthermore, let $c(v):=c_T(v)$ for all~$v\in V(T)$ and $c(w)\coloneq c_j(w)$ for all $w\in V(G\langle\mathcal{H}_j\rangle)$ and $j\in [d]$. Then $c$ witnesses that $\mathcal{H}$ is odd, a contradiction. 
    
    Thus $\mathcal{F}$ does not contain $d$ pairwise vertex-disjoint subgraphs. By \cref{HittingSet} applied to~$\mathcal{F}$, there is a set $S_i\subseteq V_i$ such that $|S_i|\leq (d-1)(w+1)$ and $G[V_i\setminus (S_i \cup \{u_i\})]$ contains no subgraph in $\mathcal{F}$. Hence $G[V_i\setminus (S\cup\{u_i\})]$ contains no non-trivial odd $U_{h-1,d}$-model. By induction, $G[V_i\setminus (S_i \cup \{u_i\})]$ can be $(3\cdot 2^{h-2}-2)$-coloured with clustering $dw+d-w$. Use a new colour for the vertices in $S_i \cup \{u_i\}$. So $G[V_i]$ can be $(3\cdot 2^{h-2}-1)$-coloured  with clustering at most $dw+d-w$. Moreover, $G[V_0]$ can be $1$-coloured with clustering $1$ since $|V_0|=1$. By alternating the colour sets used for each layer, it follows that $G$ can be coloured with $2(3\cdot 2^{h-2}-1)=3\cdot 2^{h-1}-2$ colours and with clustering at most $dw+d-w$, as required.
\end{proof}

We are now ready to prove our main result.

\begin{proof}[Proof of \cref{MainExponential}] 
    Let $G$ be an $H$-odd-minor-free graph. Then $G$ is $U_{\ctd(H),d}$-odd-minor-free for $d=|V(H)|$. By \cref{TreewidthReductionLemma}, there exists an integer $w=w(\ctd(H),d)$ such that $G$ has a red--blue colouring where each monochromatic component of $G$ has treewidth at most~$w$. Since $G$ is $U_{\ctd(H),d}$-odd-minor-free, $G$ contains no non-trivial odd $U_{\ctd(H),d}$-model. By \cref{MainTreewidth}, each monochromatic component of $G$ can be $(3\cdot 2^{\ctd(H)-1}-2)$-coloured with clustering $dw+d-w$. Using distinct colour sets for the red and blue components of $G$, we conclude that $G$ can be $(3\cdot 2^{\ctd(H)}-4)$-coloured with clustering $dw+d-w$. Hence, $\chi_{\star}(\GG_H^{\odd})\leq 3\cdot 2^{\ctd(H)}-4$.
\end{proof}

\subsection*{Acknowledgements} 
\cref{MainExponential} was first published in the Ph.D.\ thesis of the second author~\citep{KangPhD}. This result was independently discovered at the IBS-MATRIX ``\href{https://www.matrix-inst.org.au/events/structural-graph-theory-downunder-iii/}{Structural Graph Theory Downunder III}'' 
workshop at the Mathematical Research Institute MATRIX (April 2023). 

{\fontsize{10pt}{11pt}\selectfont
\bibliographystyle{DavidNatbibStyle}
\bibliography{RobReferences}}
\end{document}